\documentclass[12pt]{amsart}
\usepackage{cases}
\usepackage{mathrsfs}
\usepackage{color}
\usepackage{latexsym,amssymb}
\usepackage{a4wide}
\usepackage{amsthm}
\usepackage{amsmath}
\usepackage{graphicx}
\input xy
\xyoption{all}
\usepackage{verbatim}
\usepackage[lite,abbrev,msc-links]{amsrefs}
\usepackage{bm}

\numberwithin{equation}{section}
\begin{document}
	\theoremstyle{plain}
	\newtheorem{thm}{Theorem}[section]
	\newtheorem{lem}[thm]{Lemma}
	\newtheorem{cor}[thm]{Corollary}
	\newtheorem{cor*}[thm]{Corollary*}
	\newtheorem{prop}[thm]{Proposition}
	\newtheorem{prop*}[thm]{Proposition*}
	\newtheorem{conj}[thm]{Conjecture}	
	\theoremstyle{definition}
	\newtheorem{construction}{Construction}
	\newtheorem{notations}[thm]{Notations}
	\newtheorem{question}[thm]{Question}
	\newtheorem{prob}[thm]{Problem}
	\newtheorem{rmk}[thm]{Remark}
	\newtheorem{remarks}[thm]{Remarks}
	\newtheorem{defn}[thm]{Definition}
	\newtheorem{claim}[thm]{Claim}
	\newtheorem{assumption}[thm]{Assumption}
	\newtheorem{assumptions}[thm]{Assumptions}
	\newtheorem{properties}[thm]{Properties}
	\newtheorem{exmp}[thm]{Example}
	\newtheorem{comments}[thm]{Comments}
	\newtheorem{blank}[thm]{}
	\newtheorem{observation}[thm]{Observation}
	\newtheorem{defn-thm}[thm]{Definition-Theorem}
	\newtheorem*{Setting}{Setting}

	\newcommand{\sA}{\mathscr{A}}	\newcommand{\sB}{\mathscr{B}}
	\newcommand{\sC}{\mathscr{C}}
	\newcommand{\sD}{\mathscr{D}}
	\newcommand{\sE}{\mathscr{E}}
	\newcommand{\sF}{\mathscr{F}}
	\newcommand{\sG}{\mathscr{G}}
	\newcommand{\sH}{\mathscr{H}}
	\newcommand{\sI}{\mathscr{I}}
	\newcommand{\sJ}{\mathscr{J}}
	\newcommand{\sK}{\mathscr{K}}
	\newcommand{\sL}{\mathscr{L}}
	\newcommand{\sM}{\mathscr{M}}
	\newcommand{\sN}{\mathscr{N}}
	\newcommand{\sO}{\mathscr{O}}
	\newcommand{\sP}{\mathscr{P}}
	\newcommand{\sQ}{\mathscr{Q}}
	\newcommand{\sR}{\mathscr{R}}
	\newcommand{\sS}{\mathscr{S}}
	\newcommand{\sT}{\mathscr{T}}
	\newcommand{\sU}{\mathscr{U}}
	\newcommand{\sV}{\mathscr{V}}
	\newcommand{\sW}{\mathscr{W}}
	\newcommand{\sX}{\mathscr{X}}
	\newcommand{\sY}{\mathscr{Y}}	\newcommand{\sZ}{\mathscr{Z}}
	\newcommand{\bZ}{\mathbb{Z}}
	\newcommand{\bN}{\mathbb{N}}
	\newcommand{\bQ}{\mathbb{Q}}
	\newcommand{\bC}{\mathbb{C}}
	\newcommand{\bR}{\mathbb{R}}
	\newcommand{\bH}{\mathbb{H}}
	\newcommand{\bD}{\mathbb{D}}
	\newcommand{\bE}{\mathbb{E}}
	\newcommand{\bV}{\mathbb{V}}
	\newcommand{\cV}{\mathcal{V}}
	\newcommand{\cF}{\mathcal{F}}
	\newcommand{\bfM}{\mathbf{M}}
	\newcommand{\bfN}{\mathbf{N}}
	\newcommand{\bfX}{\mathbf{X}}
	\newcommand{\bfY}{\mathbf{Y}}
	\newcommand{\spec}{\textrm{Spec}}
	\newcommand{\dbar}{\bar{\partial}}
	\newcommand{\ddbar}{\partial\bar{\partial}}
	\newcommand{\redref}{{\color{red}ref}}
	
	\title[Koll\'ar's Package for Twisted Saito's S-sheaves] {Koll\'ar's Package for Twisted Saito's S-sheaves}
	
	\author[Junchao Shentu]{Junchao Shentu}
	\email{stjc@ustc.edu.cn}
	\address{School of Mathematical Sciences, University of Science and Technology of China, Hefei, 230026, China}
	\author[Chen Zhao]{Chen Zhao}
	\email{czhao@ustc.edu.cn}
	\address{School of Mathematical Sciences, University of Science and Technology of China, Hefei, 230026, China}
	\begin{abstract}
		We generalize Koll\'ar's conjecture (including torsion freeness, injectivity theorem, vanishing theorem and decomposition theorem) to Saito's $S$-sheaves twisted by a $\bQ$-divisor. This gives a uniform treatment for various kinds of Koll\'ar's package in different topics in complex geometry. As a consequence we prove  Koll\'ar's package of pluricanonical bundles twisted by a certain multiplier ideal sheaf. The method of the present paper is $L^2$-theoretic.
	\end{abstract}
	\maketitle
	
	\section{Introduction}
	Let $f:X\rightarrow Y$ be a proper morphism between complex spaces\footnote{All the complex spaces are assumed to be separated, reduced, paracompact, countable at infinity and of pure dimension throughout the present paper. We would like to point out that the complex spaces are allowed to be reducible.} such that $Y$ is irreducible and each irreducible component of $X$ is mapped onto $Y$. We say that a coherent sheaf $\sF$ on $X$ satisfies {\bf Koll\'ar's package} with respect to $f$ if the following statements hold.
	\begin{description}
		\item[Torsion Freeness] $R^qf_\ast (\sF)$ is torsion free for every $q\geq0$ and vanishes if $q>\dim X-\dim Y$.
		\item[Injectivity Theorem] If $L$ is a semipositive holomorphic line bundle on $X$ so that $L^{\otimes l}$ admits a nonzero holomorphic global section $s$ for some $l>0$, then the canonical morphism
		$$R^qf_\ast(\times s):R^qf_\ast(\sF\otimes L^{\otimes k})\to R^qf_\ast(\sF\otimes L^{\otimes (k+l)})$$
		is injective for every $q\geq0$ and every $k\geq1$.
		\item[Vanishing Theorem] If $Y$ is a projective algebraic variety and $L$ is an ample line bundle on $Y$, then
		$$H^q(Y,R^pf_\ast(\sF)\otimes L)=0,\quad \forall q>0,\forall p\geq0.$$
		\item[Decomposition Theorem] Assume moreover that $X$ is a compact K\"ahler space. Then $Rf_\ast (\sF)$ splits in the derived category $D(Y)$ of $\sO_Y$-modules, i.e. 
		$$Rf_\ast(\sF)\simeq \bigoplus_{q} R^qf_\ast(\sF)[-q]\in D(Y).$$
		As a consequence, the spectral sequence
		$$E^{pq}_2:H^p(Y,R^qf_\ast(\sF))\Rightarrow H^{p+q}(X,\sF)$$
		degenerates at the $E_2$ page.
	\end{description}
	These statements date back to J. Koll\'ar \cite{Kollar1986_1,Kollar1986_2}, who shows that the dualizing sheaf $\omega_X$ satisfies Koll\'ar's package when $X$ is smooth projective and $Y$ is projective. Aiming at various geometric applications, Koll\'ar's results have been further generalized in two directions.
	
	The first direction is to show Koll\'ar's package for the dualizing sheaf twisted by a $\bQ$-divisor, or a multiplier ideal sheaf more generally. These kinds of (ad-hoc) Koll\'ar's package have applications in E. Viehweg's works on the quasi-projective moduli of polarized manifolds \cite{Viehweg1995,Viehweg2010}, O. Fujino's project of minimal model program for log-canonical varieties \cite{Fujino2017} and Koll\'ar-Kov\'acs' splitting criterion for du Bois singularities \cite{Kollar2010}, etc. Besides, K. Takegoshi \cite{Takegoshi1995} proves Koll\'ar's package for the dualizing sheaf twisted by a Nakano semipositive vector bundle. The injectivity theorem for the dualizing sheaf twisted by a general multiplier ideal sheaf has been investigated by S. Matsumura \cite{Matsumura2018} and Fujino-Matsumura \cite{Matsumura2021}. However the complete proof of the Koll\'ar's package (listed as above) for the dualizing sheaf twisted by a multiplier ideal sheaf is still missing.
	
	The other direction is to generalize Koll\'ar's package to certain Hodge theoretic objects such as variations of Hodge structure and Hodge modules. Assume that $f:X\to Y$ is a morphism between projective varieties. Let $\bV$ be an $\bR$-polarized variation of Hodge structure on some dense Zariski open subset $X^o\subset X_{\rm reg}$ of the regular loci $X_{\rm reg}$. M. Saito \cite{MSaito1991} constructs a coherent sheaf $S_X(\bV)$ and shows that $S_X(\bV)$ satisfies Koll\'ar's package with respect to $f$. When $\bV$ is the trivial variation of Hodge structure, $S_X(\bV)\simeq \omega_X$. Saito's work gives an affirmative answer to Koll\'ar's conjecture \cite[\S 4]{Kollar1986_2}. Together with other deep results of Hodge modules, Koll\'ar's package for $S_X(\bV)$ plays important roles in the series works of Popa-Schnell  \cite{PS2013,PS2014,PS2017}. Recently the authors of the present paper \cite{SC2021_kollar} generalize Saito's result to the context of non-abelian Hodge theory.
	
	The purpose of the present article is to show that Koll\'ar's package holds for Saito's $S$-sheaves twisted by a multiplier ideal sheaf associated with a $\bQ$-divisor. This gives a uniform and systematic treatment for various Koll\'ar's package twisted by a $\bQ$-divisor. This package contains new results even in the case of the dualizing sheaf twisted by a multiplier ideal sheaf. The main tool is the abstract Koll\'ar's package established in \cite{SC2021_kollar}.
	\subsection{Main result}
	Before stating the main results let us recall Saito's construction of $S_X(\bV)$, with two generalizations:
	\begin{enumerate}
		\item We generalize Saito's construction to complex variations of Hodge structure. In particular we do not make assumptions on the local monodromy. This is interesting in the view of nonabelian Hodge theory because complex variations of Hodge structure are precisely the $\bC^\ast$ fixed points on the moduli space of certain tame harmonic bundles (\cite[Theorem 8]{Simpson1990}, \cite[Proposition 1.9]{Mochizuki2006}).
		\item We generalize Saito's construction with respect to the Deligne-Manin prolongations of the variation of Hodge structure with indices other than $(-1,0]$. This is a combination of Saito's $S_X(\bV)$ with the multiplier ideal sheaf associated to a boundary $\bQ$-divisor.
	\end{enumerate}
	Let $X$ be a complex space. Let $\bV=(\cV,\nabla,\{\cV^{p,q}\},Q)$ be a polarized complex variation of Hodge structure (Definition \ref{defn_CVHS}) on some dense regular Zariski open subset $X^o$ of $X$. Let $A$ be an effective $\bQ$-Cartier divisor on $X$. We define a coherent sheaf $S_X(\bV,-A)$ as follows. 
	\begin{description}
		\item[Log smooth case] Assume that $X$ is smooth, $E:=X\backslash X^o$ is a simple normal crossing divisor and ${\rm supp}(A)\subset E$. Denote by $E=\cup_{i=1}^lE_i$ the irreducible decomposition and denote $A=\sum_{i=1}^l{r_i}E_i$ with $r_1,\dots,r_l\in\bQ_{\geq 0}$. Let $\bm{r}=(r_1,\dots,r_l)$. Let $\cV_{>\bm{r}-1}$ be the Deligne-Manin prolongation with indices $>\bm{r}-1$. It is a locally free $\sO_X$-module extending $\cV$ such that $\nabla$ induces a connection with logarithmic singularities 
		$$\nabla:\cV_{>\bm{r}-1}\to\cV_{>\bm{r}-1}\otimes\Omega_X(\log E)$$ where the real part of the eigenvalues of the residue of $\nabla$ along $E_i$ belongs to $(r_i-1,r_i]$ for each $i$. Let $j:X^o\to X$ be the open immersion. Denote $S(\bV):=\cV^{p_{\rm max},k-p_{\rm max}}$ where $p_{\rm max}=\max\{p|\cV^{p,k-p}\neq0\}$. 
		Define $$S_{X}(\bV,-A):=\omega_X\otimes\left(j_\ast S(\bV)\cap\cV_{>\bm{r}-1}\right).$$ 
		\item[General case] Let $\pi:\widetilde{X}\to X$ be a proper bimeromorphic morphism such that $\pi^o:=\pi|_{\pi^{-1}(X^o\backslash{\rm supp}(A))}:\pi^{-1}(X^o\backslash{\rm supp}(A))\to X^o\backslash{\rm supp}(A)$ is biholomorphic and the exceptional loci $E:=\pi^{-1}((X\backslash X^o)\cup {\rm supp}(A)))$ is a simple normal crossing divisor. Then
		\begin{align}
			S_X(\bV,-A)\simeq\pi_\ast\left(S_{\widetilde{X}}(\pi^{o\ast}\bV,-\pi^\ast A)\right).
		\end{align}
	\end{description}
	When $A=\emptyset$, $S_X(\bV,\emptyset)$ is canonically isomorphic to Saito's $S_X(\bV)$ (see \cite{MSaito1991}, at least when $\bV$ is $\bR$-polarizable). The main result of the present article is
	\begin{thm}\label{thm_main_CVHS}
		\begin{enumerate}
			\item $S_{X}(\bV,-A)$ is a torsion free coherent sheaf on $X$ which is independent of the choice of the desingularization $\pi:\widetilde{X}\to X$.
			\item Let $f:X\to Y$ be a locally K\"ahler proper morphism between complex spaces such that $Y$ is irreducible and each irreducible component of $X$ is mapped onto $Y$. Let $L$ be a line bundle on $X$ such that some multiple $mL=B+D$ where $B$ is a semipositive line bundle and $D$ is an effective Cartier divisor on $X$. Let $F$ be an arbitrary Nakano-semipositive vector bundle on $X$. Then $S_{X}(\bV,-\frac{1}{m}D)\otimes F\otimes L$ satisfies Koll\'ar's package with respect to $f$.
		\end{enumerate}
	\end{thm}
	\subsection{Multiplier Grauert-Riemenschneider sheaf}
	When $\bV=\bC_{X_{\rm reg}}$ is the trivial variation of Hodge structure, $S_X(\bC_{X_{\rm reg}},-A)$ is exactly the Grauert-Riemenschneider sheaf twisted by the multiplier ideal sheaf associated with $A$. This is called the multiplier ideals by Viehweg \cite{Viehweg1995,Viehweg2010} and it also appear in the Nadel vanishing theorem on complex spaces \cite{Demailly2012}. Let us recall its construction for the convenience of readers.
	\begin{description}
		\item[Log smooth case] Assume that $X$ is smooth and ${\rm supp}(A)$ is a simple normal crossing divisor. Then
		$$\sK_X(-A):=\omega_X\otimes\sO_X(-\lfloor A\rfloor).$$ 
		\item[General case] Let $\pi:\widetilde{X}\to X$ be a proper bimeromorphic morphism such that $\pi^o:=\pi|_{\pi^{-1}(X^o\backslash{\rm supp}(A))}:\pi^{-1}(X^o\backslash{\rm supp}(A))\to X^o\backslash{\rm supp}(A)$ is biholomorphic and the exceptional loci $E:=\pi^{-1}((X\backslash X^o)\cup {\rm supp}(A)))$ is a simple normal crossing divisor. Then
		\begin{align}
			\sK_X(-A):=\pi_\ast\left(\sK_{\widetilde{X}}(-\pi^\ast A)\right).
		\end{align}
	\end{description}
	Certainly $\sK_X(-A)\simeq S_X(\bC_{X_{\rm reg}},-A)$ and one has 
	$$\sK_X(-A)\simeq\omega_X\otimes\sI(-A)$$
	when $X$ is smooth ($\sI(-A)$ is the multiplier ideal sheaf associated with $A$).
	In this case, by Theorem \ref{thm_main_CVHS} one has the following.
	\begin{thm}\label{thm_main_dualizing_sheaf}
		Let $f:X\to Y$ be a locally K\"ahler proper morphism between complex spaces, such that $Y$ is irreducible and each irreducible component of $X$ is mapped onto $Y$. Let $L$ be a line bundle such that some multiple $mL=B+D$ where $B$ is a semipositive line bundle and $D$ is an effective Cartier divisor on $X$. Let $F$ be an arbitrary Nakano-semipositive vector bundle on $X$. Then $\sK_{X}(-\frac{1}{m}D)\otimes F\otimes L$ satisfies Koll\'ar's package with respect to $f$.
	\end{thm}
	Theorem \ref{thm_main_dualizing_sheaf} has an application to  Koll\'ar's package of pluricanonical bundles.
	\begin{cor}
		Let $f:X\to Y$ be a morphism from a compact K\"ahler manifold to an analytic space. Assume that  $\omega_X^{\otimes km}\simeq A\otimes\sO_X(D)$, $k,m>0$ where $A$ is a semipositive line bundle (e.g. a semiample line bundle) and $D$ is an effective Cartier divisor.
		Let $F$ be an arbitrary Nakano-semipositive vector bundle on $X$. Then $\sK_{X}(-\frac{1}{m}D)\otimes\omega^{\otimes k}_X\otimes F$ satisfies Koll\'ar's package with respect to $f$. In particular if $\omega_X$ is semipositive, then $\omega_X^{\otimes k}\otimes F$ satisfies Koll\'ar's package with respect to $f$ for every $k\geq1$.
	\end{cor}
	\section{Abstract Koll\'ar's package}
	In this section we recall the abstract Koll\'ar's package established in \cite{SC2021_kollar}.
	
	Let $X$ be a complex space of dimension $n$ and $X^o\subset X_{\rm reg}$ a dense Zariski open subset. Let $(E,h)$ be a hermitian vector bundle on $X^o$.
	Define the $\sO_X$-module $S_X(E,h)$ as follows. Let $U\subset X$ be an open subset.  $S_X(E,h)(U)$ is the space of holomorphic $E$-valued $(n,0)$-forms $\alpha$ on $U\cap X^o$ such that for every point $x\in U$, there is a neighborhood $V_x$ of $x$ so that 
	$$\int_{V_x\cap X^o}\alpha\wedge\overline{\alpha}<\infty.$$
	\begin{lem}(Functoriality,\cite[Proposition 2.5]{SC2021_kollar})\label{prop_L2ext_birational}
		Let $\pi:X'\to X$ be a proper holomorphic map between complex spaces which is biholomorphic over $X^o$. Then $$\pi_\ast S_{X'}(\pi^\ast E,\pi^\ast h)=S_X(E,h).$$
	\end{lem}
	\begin{lem}(\cite[Lemma 2.6]{SC2021_kollar})\label{lem_Kernel}
		Let $(F,h_F)$ be a hermitian vector bundle on $X$ (in particular $h_F$ is smooth on $X$). Then
		$$S_X(E,h)\otimes F\simeq S_X(E\otimes F,h\otimes h_{F}).$$
	\end{lem}
	\begin{defn}\label{defn_tame_hermitian_bundle}
		$(E,h)$ is tame on $X$ if, for every point $x\in X$, there is an open neighborhood $U$ containing $x$, a proper bimeromorphic morphism $\pi:\widetilde{U}\to U$ which is biholomorphic over $U\cap X^o$, and a hermitian vector bundle $(Q,h_Q)$ on $\widetilde{U}$ such that 
		\begin{enumerate}
			\item $\pi^\ast E|_{\pi^{-1}(X^o\cap U)}\subset Q|_{\pi^{-1}(X^o\cap U)}$ as a subsheaf.
			\item There is a hermitian metric $h'_Q$ on $Q|_{\pi^{-1}(X^o\cap U)}$ so that $h'_Q|_{\pi^\ast E}\sim \pi^\ast h$ on $\pi^{-1}(X^o\cap U)$ and
			\begin{align}\label{align_tame}
				(\sum_{i=1}^r\|\pi^\ast f_i\|^2)^ch_Q\lesssim h'_Q
			\end{align}
			for some $c\in\bR$. Here $\{f_1,\dots,f_r\}$ is an arbitrary set of local generators of the ideal sheaf defining $\widetilde{U}\backslash \pi^{-1}(X^o)\subset \widetilde{U}$.
		\end{enumerate}
	\end{defn}
	The tameness condition (\ref{align_tame}) is independent of the choice of the set of local generators. In the present paper, a tame hermitian vector bundle $(E,h)$ is constructed as a subsheaf of the underlying holomorphic bundle of a variation of Hodge structure. This is a special case of tame harmonic bundles in the sense of Simpson \cite{Simpson1990} and Mochizuki \cite{Mochizuki20072,Mochizuki20071}. In this case, Condition (\ref{align_tame}) comes from the theory of degeneration of variation of Hodge structure \cite{Cattani_Kaplan_Schmid1986}.
	\begin{thm}(\cite[Proposition 2.9 and \S 4]{SC2021_kollar})\label{thm_abstract_Kollar_package}
		Let $f:X\rightarrow Y$ be a proper locally K\"ahler morphism from a complex space $X$ to an irreducible complex space $Y$. Assume that every irreducible component of $X$ is mapped onto $Y$, $X^o\subset X_{\rm reg}$ is a dense Zariski open subset and $(E,h)$ is a hermitian vector bundle on $X^o$ with Nakano semipositive curvature. Assume that $(E,h)$ is tame on $X$. Then $S_X(E,h)$ is a coherent sheaf which satisfies Koll\'ar's package with respect to $f:X\to Y$.
	\end{thm}
	\section{Twisted Saito's S-sheaf and its Koll\'ar package}
	\subsection{Complex variation of Hodge structure}
	\begin{defn}{\cite[\S 8]{Simpson1988}}\label{defn_CVHS}
	Let $X^o$ be a complex manifold. Denote by $\sA^0_{X^o}$ the sheaf of $C^\infty$ functions on $X^o$.
		A polarized complex variation of Hodge structure on $X^o$ of weight $k$ is a flat holomorphic connection $(\cV,\nabla)$ on $X^o$ together with a decomposition $\cV\otimes_{\sO_{X^o}}\sA^0_{X^o}=\bigoplus_{p+q=k}\cV^{p,q}$ of $C^\infty$ bundles and a flat hermitian form $Q$ on $\cV$ such that
		\begin{enumerate}
			\item The hermitian form $h_Q$ which equals $(-1)^{p}Q$ on $\cV^{p,q}$ is a hermitian metric on the $C^\infty$ complex vector bundle $\cV\otimes_{\sO_{X^o}}\sA^0_{X^o}$.
			\item The decomposition $\cV\otimes_{\sO_{X^o}}\sA^0_{X^o}=\bigoplus_{p+q=k}\cV^{p,q}$ is orthogonal with respect to $h_Q$.
			\item The Griffiths transversality condition 
			\begin{align}\label{align_Griffiths transversality}
				\nabla(\cV^{p,q})\subset \sA^{0,1}(\cV^{p+1,q-1})\oplus \sA^{1,0}(\cV^{p,q})\oplus\sA^{0,1}(\cV^{p,q})\oplus \sA^{1,0}(\cV^{p-1,q+1})
			\end{align}
			holds for every $p$ and $q$. Here $\sA^{i,j}(\cV^{p,q})$ denotes the sheaf of smooth $(i,j)$-forms with values in $\cV^{p,q}$.
		\end{enumerate}
		Denote $S(\bV):=\cV^{p_{\rm max},k-p_{\rm max}}$ where $p_{\rm max}=\max\{p|\cV^{p,k-p}\neq0\}$. 
	\end{defn}
	Let $X$ be a complex manifold and $\cup_{i=1}^l E_i=E:=X\backslash X^o\subset X$ a simple normal crossing divisor where $E_1,\dots, E_l$ are irreducible components. Let $(\cV,\nabla,\{\cV^{p,q}\},Q)$ be a polarized complex variation of Hodge structure on $X^o:=X\backslash E$. There is a system of prolongations of $\cV$. Let $\bm{a}=(a_1,\dots,a_l)\in\bR^l$.
	Let $\cV_{>\bm{a}}$ be the Deligne-Manin prolongation with indices $>\bm{a}$. It is a locally free $\sO_X$-module extending $\cV$ such that $\nabla$ induces a connection with logarithmic singularities 
	$$\nabla:\cV_{>\bm{a}}\to\cV_{>\bm{a}}\otimes\Omega_X(\log E)$$ whose real part of the eigenvalues of the residue of $\nabla$ along $E_i$ belongs to $(a_i,a_i+1]$. Denote 
	$$R_X(\bV):=\cV_{>\bm{-1}}\cap j_\ast(S(\bV))$$
	where $j:X^o\to X$ is the open immersion and $\bm{-1}=(-1,\dots,-1)$. By the nilpotent orbit theorem \cite{Cattani_Kaplan_Schmid1986} $R_X(\bV)$ is a subbundle of $\cV_{>\bm{-1}}$, i.e. both $R_X(\bV)$ and $\cV_{>\bm{-1}}/R_X(\bV)$ are locally free.
	\subsection{$L^2$-adapted local frame on $R_X(\bV)$}
	Let $\bV=(\cV,\nabla,\{\cV^{p,q}\},Q)$ be a polarized complex variation of Hodge structure over $(\Delta^\ast)^n\times \Delta^m$. Denote by $h_Q$ the associated Hodge metric. Let $s_1,\dots,s_n$ be holomorphic coordinates of $(\Delta^\ast)^n$ and denote $D_i:=\{s_i=0\}\subset\Delta^{n+m}$. Let $N_i$ be the unipotent part of ${\rm Res}_{D_i}\nabla$ and let 
	$$p:\bH^{n}\times \Delta^m\to (\Delta^\ast)^n\times \Delta^m,$$ 
	$$(z_1,\dots,z_n,w_1,\dots,w_m)\mapsto(e^{2\pi\sqrt{-1}z_1},\dots,e^{2\pi\sqrt{-1}z_n},w_1,\dots,w_m)$$
	be the universal covering. Let
	$W^{(1)}=W(N_1),\dots,W^{(n)}=W(N_1+\cdots+N_n)$ be the monodromy weight filtrations (centered at 0) on $V:=\Gamma(\bH^n\times \Delta^m,p^\ast\cV)^{p^\ast\nabla}$.
	The following norm estimate for flat sections is proved by Cattani-Kaplan-Schmid \cite[Theorem 5.21]{Cattani_Kaplan_Schmid1986} for the case when $\bV$ has quasi-unipotent local monodromy and by Mochizuki \cite[Part 3, Chapter 13]{Mochizuki20072} for the general case.
	\begin{thm}\label{thm_Hodge_metric_asymptotic}
		For any $0\neq v\in {\rm Gr}_{l_n}^{W^{(n)}}\cdots{\rm Gr}_{l_1}^{W^{(1)}}V$, one has
		\begin{align*}
		|v|^2_{h_Q}\sim \left(\frac{\log|s_1|}{\log|s_2|}\right)^{l_1}\cdots\left(-\log|s_n|\right)^{l_n}
		\end{align*}
		over any region of the form
		$$\left\{(s_1,\dots s_n,w_1,\dots,w_m)\in (\Delta^\ast)^n\times \Delta^m\bigg|\frac{\log|s_1|}{\log|s_2|}>\epsilon,\dots,-\log|s_n|>\epsilon,(w_1,\dots,w_m)\in K\right\}$$
		for any $\epsilon>0$ and an arbitrary compact subset $K\subset \Delta^m$ .
	\end{thm}
    The rest of this part is devoted to the norm estimate of the singular hermitian metric $h_Q$ on $R_X(\bV)$. 
	\begin{lem}\label{lem_W_F}
		Assume that $n=1$. Then $W_{-1}(N_1)\cap R_X(\bV)_{\bf 0}=0$.
	\end{lem}
	\begin{proof}
		Assume that $W_{-1}(N_1)\cap R_X(\bV)_{\bf 0}\neq0$. Let $k$ be the weight of $\bV$. Let $l=\max\{l|W_{-l}(N_1)\cap R_X(\bV)_{\bf 0}\neq0\}$. Then $l\geq 1$. 
		By \cite[6.16]{Schmid1973}, the decomposition $\cV_{>\bm{-1}}\simeq \bigoplus_{p+q=k}j_\ast \cV^{p,q}\cap\cV_{>\bm{-1}}$ induces a pure Hodge structure of weight $m+k$ on $W_{m}(N_1)/W_{m-1}(N_1)$. Moreover 
		\begin{align}\label{align_hard_lef_N}
		N^l: W_{l}(N_1)/W_{l-1}(N_1)\to W_{-l}(N_1)/W_{-l-1}(N_1)
		\end{align}
		is an isomorphism of type $(-l,-l)$. Denote $S(\bV)=\cV^{p,k-p}$. By the definition of $l$, any nonzero element $\alpha\in W_{-l}(N_1)\cap R_X(\bV)_{\bf 0}$ induces a nonzero $[\alpha]\in W_{-l}(N_1)/W_{-l-1}(N_1)$ of Hodge type $(p,k-l-p)$. Since (\ref{align_hard_lef_N}) is an isomorphism, there is $\beta\in W_{l}(N_1)/W_{l-1}(N_1)$ of Hodge type $(p+l,k-p)$ such that $N^l(\beta)=[\alpha]$. However, $\beta=0$ since $\cF^{p+l}=0$. This contradicts to the fact that $[\alpha]\neq0$. $W_{-1}(N_1)\cap R_X(\bV)_{\bf 0}$ therefore has to be zero.
	\end{proof}
	Denote by $T_i$ the local monodromy operator of $\bV$ around $D_i$.
	Since $T_1,\dots,T_n$ are pairwise commutative, there is a finite decomposition 
	$$\cV_{>\bm{-1}}|_{\bf 0}=\bigoplus_{-1<\alpha_1,\dots,\alpha_n\leq 0}\bV_{\alpha_1,\dots,\alpha_n}$$
	such that $(T_i-e^{2\pi\sqrt{-1}\alpha_i}{\rm Id})$ is unipotent on $\bV_{\alpha_1,\dots,\alpha_n}$ for each $i=1,\dots,n$. 
	
	Let $$v_1,\dots, v_N\in R_X(\bV)|_{\bf 0}\cap\bigcup_{-1<\alpha_1,\dots,\alpha_n\leq 0}\bV_{\alpha_1,\dots,\alpha_n}$$
	be an orthogonal basis of $R_X(\bV)|_{\bf 0}\simeq \Gamma(\bH^n,p^\ast S(\bV))^{p^\ast\nabla}$. Then $\widetilde{v_1},\dots,\widetilde{v_N}$ that are determined by
	\begin{align}\label{align_adapted_frame}
	\widetilde{v_j}:={\rm exp}\left(\sum_{i=1}^n\log z_i(\alpha_i{\rm Id}+N_i)\right)v_j\textrm{ if } v_j\in\bV_{\alpha_1,\dots, \alpha_n},\quad \forall j=1,\dots,N
	\end{align}
	form a frame of $\cV_{>\bm{-1}}\cap j_\ast S(\bV)$.
	To be precise, we always use the notation $\alpha_{E_i}(\widetilde{v_j})$ instead of $\alpha_i$ in (\ref{align_adapted_frame}). By (\ref{align_adapted_frame}) we acquire that 
	\begin{align*}
	|\widetilde{v_j}|^2_{h_Q}&\sim\left|\prod_{i=1}^nz_i^{\alpha_{E_i}(\widetilde{v_j})}{\rm exp}\left(\sum_{i=1}^nN_i\log z_i\right)v_j\right|^2_{h_Q}\\\nonumber
	&\sim|v_j|^2_{h_Q}\prod_{i=1}^n |z_i|^{2\alpha_{E_i}(\widetilde{v_j})},\quad j=1,\dots,N
	\end{align*}
	where $\alpha_{E_i}(\widetilde{v_j})\in(-1,0]$, $\forall i=1,\dots n$. 
	
	By Theorem \ref{thm_Hodge_metric_asymptotic} and Lemma \ref{lem_W_F}
	one has
	\begin{align*}
	|v_j|^2_{h_Q}\sim \left(\frac{\log|s_1|}{\log|s_2|}\right)^{l_1}\cdots\left(-\log|s_n|\right)^{l_n},\quad l_1\leq l_2\leq\dots\leq l_{n-1},
	\end{align*}
	over any region of the form
	$$\left\{(s_1,\dots s_n,w_1,\dots,w_{m})\in (\Delta^\ast)^n\times \Delta^{m}\bigg|\frac{\log|s_1|}{\log|s_2|}>\epsilon,\dots,-\log|s_n|>\epsilon,(w_1,\dots,w_{m})\in K\right\}$$
	for any $\epsilon>0$ and an arbitrary compact subset $K\subset \Delta^{m}$. Therefore we obtain that
	\begin{align*}
	1\lesssim |v_j|\lesssim|z_1\cdots z_r|^{-\epsilon},\quad\forall\epsilon>0.
	\end{align*}
	
	The local frame $(\widetilde{v_1},\dots,\widetilde{v_N})$ is $L^2$-adapted in the sense of S. Zucker \cite[page 433]{Zucker1979}.
	\begin{defn}
		Let $(E,h)$ be a vector bundle with a possibly singular hermitian metric $h$ on a hermitian manifold $(X,ds^2_0)$. A holomorphic local frame $(v_1,\dots,v_N)$ of $E$ is called $L^2$-adapted if, for every set of measurable functions $\{f_1,\dots,f_N\}$, $\sum_{i=1}^Nf_iv_i$ is locally square integrable if and only if $f_iv_i$ is locally square integrable for each $i=1,\dots,N$.
	\end{defn}
	To see that $(\widetilde{v_1},\dots,\widetilde{v_N})$ is $L^2$-adapted, let us consider the measurable functions $f_1,\dots,f_N$. If 
	$$\sum_{j=1}^N f_j\widetilde{v_j}={\rm exp}\left(\sum_{i=1}^nN_i\log z_i\right)\left(\sum_{j=1}^N f_j\prod_{i=1}^n |z_i|^{\alpha_{E_i}(\widetilde{v_j})}v_j\right)$$
	is locally square integrable, then 
	$$\sum_{j=1}^N f_j\prod_{i=1}^n |z_i|^{\alpha_{E_i}(\widetilde{v_j})}v_j$$
	is locally square integrable because the entries of the matrix ${\rm exp}\left(-\sum_{i=1}^nN_i\log z_i\right)$ are $L^\infty$-bounded.
	Since $(v_1,\dots,v_N)$ is an orthogonal basis, 
	$|f_j\widetilde{v_j}|_{h_Q}\sim\prod_{i=1}^n |z_i|^{\alpha_{E_i}(\widetilde{v_j})}|f_jv_j|_{h_Q}$ is locally square integrable for each $j=1,\dots,N$. 
	
	In conclusion, we obtain the following proposition.
	\begin{prop}\label{prop_adapted_frame}
		Let $(X,ds^2_0)$ be a hermitian manifold and $E$ a normal crossing divisor on $X$. Let $\bV$ be a polarized complex variation of Hodge structure on $X^o:=X\backslash E$. Then there is an $L^2$-adapted holomorphic local frame $(\widetilde{v_1},\dots,\widetilde{v_N})$ of $R_X(\bV)$ at every point $x\in E$. There are moreover $\alpha_{E_i}(\widetilde{v_j})\in(-1,0]$, $i=1,\dots, r$, $j=1,\dots,N$ and positive real functions $\lambda_j\in C^\infty(X\backslash E)$, $j=1,\dots,N$ such that
		\begin{align}\label{align_L2adapted_frame}
		|\widetilde{v_j}|^2\sim\lambda_j\prod_{i=1}^r |z_i|^{2\alpha_{E_i}(\widetilde{v_j})},\quad \forall j=1,\dots,N
		\end{align}
		and
		$$1\lesssim \lambda_j\lesssim|z_1\cdots z_r|^{-\epsilon},\quad\forall\epsilon>0$$
		for each $j=1,\dots,N$.
		Here $z_1,\cdots,z_n$ are holomorphic local coordinates on $X$ so that  $E_i=\{z_i=0\}$, $i=1,\cdots r$ and $E
		=\{z_1\cdots z_r=0\}$.
	\end{prop}
	\subsection{Twisted Saito's S-sheaf}
	Let $X$ be a complex space and $X^o\subset X_{\rm reg}$ a dense Zariski open subset. Let $\bV=(\cV,\nabla,\{\cV^{p,q}\},Q)$ be a polarized complex variation of Hodge structure on $X^o$. Let $A$ be an effective $\bQ$-Cartier divisor on $X$. We define a coherent sheaf $S_X(\bV,-A)$ as follows. 
	\begin{description}
		\item[Log smooth case] Assume that $X$ is smooth, $E:=X\backslash X^o$ is a simple normal crossing divisor and ${\rm supp}(A)\subset E$. Denote by $E=\cup_{i=1}^lE_i$ the irreducible decomposition and denote $A=\sum_{i=1}^l{r_i}E_i$ with $r_1,\dots,r_l\in\bQ_{\geq 0}$. Let $\bm{r}=(r_1,\dots,r_l)$. Let $\cV_{>\bm{r}-1}$ be the Deligne-Manin prolongation with indices $>\bm{r}-1$. It is a locally free $\sO_X$-module extending $\cV$ such that $\nabla$ induces a connection with logarithmic singularities 
		$$\nabla:\cV_{>\bm{r}-1}\to\cV_{>\bm{r}-1}\otimes\Omega_X(\log E)$$ where the real part of the eigenvalues of the residue of $\nabla$ along $E_i$ belongs to $(r_i-1,r_i]$ for each $i$. Let $j:X^o\to X$ be the open immersion. Denote $S(\bV):=\cV^{p_{\rm max},k-p_{\rm max}}$ where $p_{\rm max}=\max\{p|\cV^{p,k-p}\neq0\}$. 
		Define $$S_{X}(\bV,-A):=\omega_X\otimes\left(j_\ast S(\bV)\cap\cV_{>\bm{r}-1}\right).$$ 
		\item[General case] Let $\pi:\widetilde{X}\to X$ be a proper bimeromorphic morphism such that $\pi^o:=\pi|_{\pi^{-1}(X^o\backslash{\rm supp}(A))}:\pi^{-1}(X^o\backslash{\rm supp}(A))\to X^o\backslash{\rm supp}(A)$ is biholomorphic and the exceptional loci $E:=\pi^{-1}((X\backslash X^o)\cup {\rm supp}(A)))$ is a simple normal crossing divisor. Then
		\begin{align}
			S_X(\bV,-A)\simeq\pi_\ast\left(S_{\widetilde{X}}(\pi^{o\ast}\bV,-\pi^\ast A)\right).
		\end{align}
	\end{description}
	Let $L$ be a line bundle such that some multiple $mL=B+D$ where $B$ is a semipositive line bundle and $D$ is an effective Cartier divisor on $X$.
	Let $h_B$ be a hermitian metric on $B$ with semipositive curvature and $h_D$ the unique singular hermitian metric on $\sO_X(D)$ determined by the effective divisor $D$.  $h_D$ is a singular hermitian metric, smooth over $X\backslash D$, defined as follows. Let $s\in H^0(X,\sO_X(D))$ be the defining section of $D$ and let $h_0$ be an arbitrary smooth hermitian metric on $\sO_X(D)$. Then $h_D$ is defined by $|\xi|_{h_D}=|\xi|_{h_0}/|s|_{h_0}$ which is independent of the choice of $h_0$.
	Denote $h_L:=(h_Bh_D)^{\frac{1}{m}}$.
	The main result of this section is
	\begin{thm}\label{thm_L2_interpretation}
		$S_X(\bV,-\frac{1}{m}D)\otimes L\simeq S_X(S(\bV)\otimes L,h_Qh_L)$. In particular $S_X(\bV,-\frac{1}{m}D)$ is independent of the choice of the desingularization $\pi:\widetilde{X}\to X$. 
	\end{thm}
	\begin{proof}
		By Lemma \ref{prop_L2ext_birational}, the proof can be reduced to the log smooth case, that is, $X$ is smooth, $E:=X\backslash X^o$ is a simple normal crossing divisor and ${\rm supp}(D)\subset E$. Denote $j:X^o:=X\backslash E\to X$ to be the inclusion. We are going to show that $$S_X(\bV,-\frac{1}{m}D)\otimes L=S_X(S(\bV)\otimes L,h_Qh_L)$$ as subsheaves of $\omega_X\otimes j_\ast(S(\bV))\otimes L$.
		Since the problem is local, we assume that $X=\Delta^n$ is the polydisc. Denote $E:=\{z_1\cdots z_l=0\}$ where  $E_i:=\{z_i=0\}$ for each $i=1,\dots,l$.  Let $\bV=(\cV,\nabla,\{\cV^{p,q}\},h_Q)$ be a polarized complex variation of Hodge structure on $X^o$. Let ${\bf 0}=(0,\dots,0)\in X$ and let $(\widetilde{v_1},\dots,\widetilde{v_N})$ be an $L^2$-adapted local frame of $R_X(\bV)$ at ${\bf 0}$ as in Proposition \ref{prop_adapted_frame}. Let $f_1,\dots,f_N\in (j_\ast\sO_{X^o})_{\bf 0}$ and let $e$ be the local frame of $L$ at $\bm{0}$. We are going to prove that
		$$\sum_{i=1}^Nf_i[\widetilde{v_i}dz_1\wedge\cdots\wedge dz_n\otimes e]_{\bf 0}\in S_X(S(\bV)\otimes L,h_Qh_L)_{\bf 0}$$ if and only if
		$$f_i\in\sO_X\big(-\sum_{j=1}^l\lfloor\frac{r_i}{m}-\alpha_{E_j}(\widetilde{v_i})\rfloor E_j\big)_{\bf 0}$$
		for every $i=1,\dots,N$. 
		
		Denote $ds^2_0=\sum_{i=1}^ndz_id\bar{z}_i$.
		Since $(\widetilde{v_1},\dots,\widetilde{v_N})$ is an $L^2$-adapted frame as in Proposition \ref{prop_adapted_frame}, the integral
		$$\int|\sum_{i=1}^Nf_i\widetilde{v_i}dz_1\wedge\cdots\wedge dz_n|^2|e|^2_{h_L}{\rm vol}_{ds^2_0}=\int|\sum_{i=1}^Nf_i\widetilde{v_i}|^2|e|^2_{h_L}{\rm vol}_{ds^2_0}$$ is finite near ${\bf 0}$ if and only if
		\begin{align}\label{align_l2_term}
			\int|f_i\widetilde{v_i}|^2|e|^2_{h_L}{\rm }{\rm vol}_{ds^2_0}\sim \int|f_i|^2{\rm }\prod_{j=1}^r|z_j|^{2\alpha_{E_j}(\widetilde{v_i})-\frac{2r_i}{m}}\lambda_i{\rm vol}_{ds^2_0}
		\end{align}
		is finite near ${\bf 0}$ for every $i=1,\dots,N$. Here $\lambda_i$ is a positive real function so that
		\begin{align}\label{align_keylem_lambda}
			1\lesssim \lambda_i\lesssim |z_1\cdots z_r|^{-\epsilon},\quad\forall\epsilon>0.
		\end{align}	
		Denote 
		$$v_j(f):=\min\{l|f_l\neq0\textrm{ in the Laurant expansion } f=\sum_{i\in\bZ}f_iz_j^i\}.$$
		By Lemma \ref{lem_integral}, the local integrability of (\ref{align_l2_term}) is equivalent to that
		\begin{align}
			v_j( f_i)+\alpha_{E_i}(\widetilde{v_j})-\frac{r_i}{m}>-1,\quad \forall j=1,\dots, l.
		\end{align}
		This is equivalent to 
		\begin{align}
			v_j( f_i)\geq\lfloor-\alpha_{E_i}(\widetilde{v_j})+\frac{r_i}{m}\rfloor,\quad \forall j=1,\dots, l.
		\end{align}
		As a consequence, $S_X(S(\bV)\otimes L,h_Qh_L)_{\bf 0}$ is generated by $$dz_1\wedge\cdots\wedge dz_n\otimes e\otimes{\rm exp}\left(\sum_{i=1}^n\log z_i(\lfloor-\alpha_{E_i}(\widetilde{v_j})+\frac{r_i}{m}\rfloor{\rm Id}+N_i)\right)\widetilde{v_j} ,\quad \forall j=1,\dots,N.$$
		These are exactly the generators of $\omega_{X}\otimes L\otimes(j_\ast(S(\bV))\cap\cV_{>\frac{\bm{r}}{m}-1})$ at ${\bf 0}$.
		The proof is finished.
	\end{proof}
	The proof of the following lemma is omitted.
	\begin{lem}\label{lem_integral}
		Let $f$ be a holomorphic function on $\Delta^\ast:=\{z\in\bC|0<|z|<1\}$ and $a\in\bR$. Then
		$$\int_{|z|<\frac{1}{2}}|f|^2|z|^{2a}dzd\bar{z}<\infty$$
		if and only if $v(f)+a>-1$. Here
		$$v(f):=\min\{l|f_l\neq0\textrm{ in the Laurant expansion } f=\sum_{i\in\bZ}f_iz^i\}.$$
	\end{lem}
	\subsection{Koll\'ar package}
	In this section we prove the main theorem (Theorem \ref{thm_main_CVHS}) of the present paper.
	Let $X$ be a complex space and $X^o\subset X_{\rm reg}$ a dense Zariski open subset. Let $\bV:=(\cV,\nabla,\{\cV^{p,q}\},Q)$ be a polarized complex variation of Hodge structure of weight $k$ on $X^o$. Let
	$$\nabla=\overline{\theta}+\partial+\dbar+\theta$$
	be the decomposition according to (\ref{align_Griffiths transversality}).
	For the reason of degrees, $S(\bV)$ is a holomorphic subbundle of $\cV$ and $\overline{\theta}(S(\bV))=0$. \begin{lem}\label{lem_SV_tame}
		$(S(\bV),h_Q)$ is a Nakano semipositive vector bundle which is tame on $X$.
	\end{lem}
	\begin{proof}
		To see that $(S(\bV),h_Q)$ is Nakano semipositive, we take the decomposition 
		$$\nabla=\overline{\theta}+\partial+\dbar+\theta$$
		according to (\ref{align_Griffiths transversality}).
		Since $\overline{\theta}(S(\bV))=0$, it follows from Griffiths' curvature formula
		$$\Theta_h(S(\bV))+\theta\wedge\overline{\theta}+\overline{\theta}\wedge\theta=0$$
		that
		$$\sqrt{-1}\Theta_h(S(\bV))=-\sqrt{-1}\overline{\theta}\wedge\theta\geq_{\rm Nak} 0.$$
		To prove the tameness we use Deligne's extension. Since the problem is local, we assume that there is a desingularization  $\pi:\widetilde{X}\to X$ such that $\pi$ is biholomorphic over $X^o$ and $D:=\pi^{-1}(X\backslash X^o)$ is a simple normal crossing divisor. By abuse of notations we identify $X^o$ and $\pi^{-1}(X^o)$. There is an inclusion $S(\bV)\subset\cV_{>\bm{-1}}|_{X^o}$. Let $z_1,\dots,z_n$ be holomorphic local coordinates such that $D_i=\{z_i=0\}$, $i=1,\dots,k$ and $D=\{z_1\cdots z_k=0\}$. By Theorem \ref{thm_Hodge_metric_asymptotic}, one has the norm estimate 
		\begin{align}\label{align_normest_tame}
			|z_1\cdots z_k|^{}|s|_{h_0}\lesssim |s|_h
		\end{align}
		for any  holomorphic local  section $s$ of $\cV_{>\bm{-1}}$. Here $h_0$ is an arbitrary (smooth) hermitian metric on $\cV_{>\bm{-1}}$.	
        This shows that $(S(\bV),h_Q)$ is tame.
	\end{proof}
	\begin{thm}
		Let $X$ be a complex space and $X^o\subset X_{\rm reg}$ a dense Zariski open subset. Let $\bV:=(\cV,\nabla,\{\cV^{p,q}\},Q)$ be a polarized complex variation of Hodge structure of weight $k$ on $X^o$.
		Let $L$ be a line bundle such that some multiple $mL=A+D$ where $A$ is a semipositive line bundle and $D$ is an effective Cartier divisor on $X$. Let $F$ be an arbitrary Nakano-semipositive vector bundle on $X$. Then $S_{X}(\bV,-\frac{1}{m}D)\otimes F\otimes L$ satisfies Koll\'ar's package with respect to any locally K\"ahler proper morphism $f:X\to Y$ such that $Y$ is irreducible and each irreducible component of $X$ is mapped onto $Y$.
	\end{thm}
	\begin{proof}
		Let $h_A$ be a hermitian metric on $A$ with semipositive curvature and $h_D$ the singular hermitian metric on $\sO_X(D)$ determined by the effective divisor $D$. Denote $h_L:=(h_Ah_D)^{\frac{1}{m}}$. Then
		$$\sqrt{-1}\Theta_{h_L}(L)|_{X\backslash D}=\frac{\sqrt{-1}}{m}\Theta_{h_A}(A)|_{X\backslash D}\geq0.$$
		Hence $(L|_{U},h_L|_U)$ has semipositive curvature and is tame on $X$. Therefore by Lemma \ref{lem_SV_tame} $(S(\bV)\otimes L\otimes F|_U,h_Qh_Lh_F|_U)$ has semipositive curvature on $U=X^o\backslash {\rm supp}(D)$ and is tame on $X$. By Lemma \ref{lem_Kernel}, Theorem \ref{thm_abstract_Kollar_package} and Theorem \ref{thm_L2_interpretation} we obtain that
		$S_{X}(\bV,-\frac{1}{m}D)\otimes F\otimes L\simeq S_X(S(\bV)\otimes L\otimes F|_U,h_Qh_Lh_F|_U)$ satisfies Koll\'ar's package with respect to $f:X\to Y$.
	\end{proof}
	\bibliographystyle{plain}
	\bibliography{CGM_Kollar}
	
\end{document}